\documentclass[12pt]{article}
\usepackage{amssymb,amsmath,epsfig,amscd,eucal,psfrag,amsthm,relsize}
\usepackage{graphicx,hyperref}

\newtheorem{theorem}{Theorem}[section]

\newtheorem{lemma}[theorem]{Lemma}

\newtheorem{corollary}[theorem]{Corollary}

\theoremstyle{definition}
\newtheorem{definition}[theorem]{Definition}

\theoremstyle{remark}
\newtheorem{remark}[theorem]{Remark}
\newtheorem{example}[theorem]{Example}

\newcommand{\C}{\mathbb{C}}
\newcommand{\Hb}{\mathbb{H}}
\newcommand{\G}{\mathcal{G}}

\newcommand{\R}{\mathbb{R}}
\newcommand{\Z}{\mathbb{Z}}

\newcommand{\wh}{\widehat}

\newcommand{\curlyP}{\mathcal{P}}

\newcommand{\curlyH}{\mathcal{H}}

\input xy
\xyoption{all}

\begin{document}

\title{The profinite completion of relatively hyperbolic virtually special groups}

\author{ Pavel Zalesskii\footnote{Partially supported by CNPq and FAPDF}}

\maketitle

\begin{abstract} We give a characterization of  toral relatively hyperbolic virtually special groups in terms of the profinite completion.
We also prove a  Tits alternative for subgroups of the profinite completion $\widehat G$ of a relatively hyperbolic virtually compact special group $G$ and completely describe   finitely generated pro-$p$ subgroups of $\widehat G$.  This applies to the profinite completion of the fundamental group of a hyperbolic arithmetic manifold. We deduce that all finitely generated pro-$p$ subgroups of the congruence kernel of  a standard arithmetic lattice of $SO(n,1)$ are free pro-$p$.  
\end{abstract}

\section{Introduction}

In recent years there has been a great deal of interest in detecting properties of a group $G$ of geometric nature and in particular the fundamental group $\pi_1M$ of a manifold via its finite quotients, or more conceptually by its profinite completion (\cite{BMRS,BMRSbull, BR, LR, liu,sun, MW, WZ-16, WZ-19, Wilk}.  This motivates the study of the profinite completions $\widehat G$ and $\widehat {\pi_1M}$ of such groups. 

In \cite{WZ-16} it was shown that the geometry type of a geometric 3-manifold is detected by the profinite completion of the fundamental group of the manifold; in particular   the profinite version of Hyperbolization theorem and the Seifert conjecture were established there. Moreover, several structure results on the profinite completion of 3-manifold groups and more generally on the profinite completion of virtually compact special hyperbolic groups were proved  in \cite{WZ-16} and \cite{henry}. 

The objective of this paper is to extend  results of \cite{WZ-16} and \cite{henry} to relatively hyperbolic virtually compact special groups and  standard arithmetic hyperbolic manifolds of higher dimension.

\medskip
Virtually compact special cube complexes, developed by Wise and his collaborators are playing a central role in  modern geometric group theory.  A group is called  virtually compact  special if it is the fundamental group of a compact virtually special cube complex whose hyperplanes satisfy certain combinatorial conditions. Wise’s Quasiconvex Hierarchy Theorem \cite[Theorem 13.3]{wise_structure_2012} characterizes the virtually compact special hyperbolic groups in terms of virtual quasiconvex hierarchies.

We first observe  that the profinite version of Hyperbolization theorem   \cite[Theorem A]{WZ-16}  extends to  virtually compact special groups  using \cite[Corollary 1.2] {G}, \cite[Corollary 1.6]{M} and \cite[Theorem F]{WZ-16}. 

\begin{theorem} Let $G$ be a virtually compact special group and $p$ be a prime. Then  $G$ is hyperbolic if and only if $\widehat G$ does not contain $\Z_p\times  \Z_p$, where $\Z_p$ is the group of $p$-adic integers. 
\end{theorem}

Our first result extends this characterization to an  algebraic profinite characterization of toral relatively hyperbolic virtually compact special groups. It is inspired by the abstract analog of it due to Anthony Genevois \cite{G}.

\begin{theorem}\label{relative hypebolization} A virtually compact special group is  hyperbolic relative to a finite collection of virtually abelian subgroups if and only if its profinite completion does not contain $F_2\times \Z_p$, where $F_2$ is a free pro-$p$ group of rank 2 and $p$ is some prime.
\end{theorem}

Sylow theory  and   pro-$p$ subgroups have the same importance   in profinite group theory as Sylow theory and $p$-groups  in the theory of finite groups.  Note that pro-$p$ subgroups of $\widehat G$ are infinitely generated in general, so it is natural to begin the study of them  from the finitely generated case. 
Our next theorem   describes  finitely generated pro-$p$ subgroups of the profinite completion of a relatively hyperbolic virtually compact special group and gives the  `Tits alternative' for relatively hyperbolic virtually compact special groups that generalizes \cite[Theorem C]{WZ-16}. It is based on  relatively hyperbolic analogs of Wise’s Quasiconvex Hierarchy Theorem proved by Eduard Einstein \cite{E}.
 We use the notation $\Z_\pi$ to denote $\prod_{p\in \pi} \Z_p$ and use the term parabolic subgroup of $\widehat G$ for closed subgroups of the profinite completion of maximal parabolic subgroups.  

\begin{theorem} \label{main} Let $G$ be a   relatively hyperbolic virtually compact special group and $H$ be a closed
subgroup of $\wh{G}$.

\begin{enumerate}
\item[(i)] If $H$
  does not contain a free non-abelian pro-$p$ subgroup for any prime $p$ then either
 $H$ is parabolic or virtually  isomorphic to
$\Z_\pi\rtimes \Z_\rho$, where $\pi$ and $\rho$
are (possibly empty) sets of primes with $\pi\cap
\rho=\varnothing$ (the action can be described as multiplication by units of $\Z_\pi$ and so $\Z_\rho$ acts as a subgroup of a cyclic group $C_{p-1}$ on $\Z_p$ for odd $p\in \pi$; for  $p=2$ it acts trivially or by inversion).  If $G$ is torsion free, then
`virtually' can be omitted.

\item[(ii)] if $H$ is finitely generated pro-$p$ then $H$ is the pro-$p$ fundamental group of a finite graph of pro-$p$ groups whose edge groups are finite and vertex groups are  parabolic. In particular, if $H$ is torsion free, then $H$ is a free pro-$p$ product of parabolic subgroups.
\end{enumerate}
\end{theorem}

The proof of Theorem  \ref{main} follows the same strategy as in \cite{WZ-16}, where   the result was proved    in the hyperbolic case: by results of E. Einstein \cite{E} after passing to a finite index subgroup there is a suitable hierarchy which is separable and so it induces  the corresponding action of the profinite completion $\widehat G$  on a profinite tree which is  relatively profinitely acylindrical. 
At this point we make use of the profinite analogue of Bass--Serre theory for groups acting on trees.

\bigskip
The manifolds to which theorem above can be applied are standard arithmetic hyperbolic manifolds.
 An  arithmetic hyperbolic manifold is obtained as a quotient of $SO(n,1)$ by the action of an  arithmetic lattice.  Note that there are standard and non-standard arithmetic subgroups in 
$SO(n , 1)$.  Standard arithmetic lattices  come from quadratic forms $q$ over an algebraic number field $k$ that are  totally real and for one real valuation  $q$ is equivalent over $\R$ to the form $x_1^2  + ..+ x_n^2  - x_{n+1}^2$ and for all other valuations is either positive or negative definite.  
Let $R_k$ be the ring of integers of $k$. A {\it standard arithmetic subgroup} $\Gamma$ of $SO(n,1)$ is a group commensurable with the subgroup $SO(q, R_k)$ of $R_k$-points of $SO(q)= \{X \in SL(n + 1, \C) \mid X^t q X = q \}$. An arithmetic hyperbolic $n$-manifold 
$M = \Hb^n/\Gamma$ is 
called standard arithmetic  if $\Gamma$ is.

 Note also that for $n$ even  all arithmetic subgroups of $SO(n,1)$ are standard.   Cocompact standard arithmetic subgroups of $SO(n,1)$ are virtually compact special by  the  result of Bergeron, Haglund  and Wise \cite[Theorem 1.10]{bergeron} and this is the main reason why  we can apply  the results above to  cocompact standard arithmetic subgroups of $SO(n,1)$.
 
If the standard arithmetic subgroup $G$ is not cocompact then by \cite[Lemma 6.3 and Theorem 7.4]{PS}    
 it posseses  a finite index subgroup $G_0$ that embeds in a virtually compact special group $G'_0$  that  splits as a star of groups whose central vertex group is $G_0$ and all non-central vertex groups are finitely generated free abelian.   This means that $G_0$ is a virtual retract of $G'_0$ and we can apply Theorem  \ref{main} to $G'_0$ and then induce the obtained subgroup structure on $G_0$.
 
  Non-standard arithmetic lattices  come from Hermitian forms and  are not commensurable with standard ones; it is unknown whether they are virtually compact special, so they will not be considered in this paper.

 It is well-known that the fundamental group of a hyperbolic manifold is hyperbolic relative to virtually abelian subgroups. Thus applying Theorem  \ref{main} we deduce

\begin{theorem}\label{thm: 3M Tits alternative}
Let M be any standard hyperbolic arithmetic manifold (possibly with cusps) and $H$  a closed subgroup of $\wh{\pi_1M}$.

\begin{enumerate}
\item[(i)] If $H$
  does not contain a free non-abelian pro-$p$ subgroup for any prime $p$ then either
 $H$ is abelian or   isomorphic to
$\Z_\pi\rtimes \Z_\rho$, where $\pi$ and $\rho$
are (possibly empty) sets of primes with $\pi\cap
\rho=\varnothing$ ( and the action is the multiplication by units).

\item[(ii)] if $H$ is finitely generated pro-$p$ then $H$ is a free
pro-$p$ product of virtually free abelian pro-$p$ groups. Moreover, each non-cyclic virtually abelian free factor of $H$ is conjugate into the profinite completion of a cusp subgroup of $\pi_1M$.
\end{enumerate}
\end{theorem}

A standard arithmetic subgroup $\Gamma$ of $SO(n,1)$ might have torsion. In this case  we use the main result of \cite{WEZ-16} to deduce the following description of
  finitely generated pro-$p$ subgroups of $\widehat \Gamma$.

\begin{corollary}\label{pro-p of arithmetic} Let $\Gamma$ be a standard arithmetic lattice of $SO(n,1)$ and $H$ a finitely generated pro-$p$ subgroup of $\widehat \Gamma$. Then $H$ is the pro-$p$  fundamental group of a finite graph of  groups whose edge groups are finite $p$-groups and vertex groups are abelian-by-finite pro-$p$-groups.\end{corollary}

If $\Gamma$ is an arithmetic subgroup of $SO(n,1)$ then one can define the
{\it congruence} topology by taking
the subgroups
$$\Gamma(\alpha)=\{g\in {\bf SL_{n+1}}({R_k})\cap \Gamma\}\mid g\equiv 1 ({\rm mod}\
\alpha)\}$$ corresponding to non-zero ideals $\alpha$ of the ring of integers $R_k$ as
basis of neighbourhoods of the identity.  The completion
$\overline\Gamma$ of $\Gamma$ with respect to this topology is called  the
{\it congruence completion}.
The {\it congruence kernel}
 $C=C(\Gamma)$ is the kernel of
the natural epimorphism
 $\widehat{\Gamma}\longrightarrow \overline\Gamma$.  The congruence subgroup problem in modern understanding (due to J.-P. Serre) is stated as follows:
 
 \bigskip

 {\bf The congruence subgroup
problem} :  Compute or describe the congruence kernel $C$.

\bigskip
The books of Margulis \cite[p. 268]{Mar} and Platonov-Rapinchuk \cite[ Section 9.5]{PR} emphasize the importance of determining the structure of the congruence kernel (A. Lubotzky refers to this as the complete solution of the congruence subgroup problem.)

Bass, Milnor, Serre \cite{BMS} described $C$
for ${\bf G}=SL_n (n\geq 3)$ and $Sp_{2n}(n\geq 2)$. When $C$ is finite, it is central and a general result on the congruence subgroup problem for isotropic groups  is due to Raghunatan \cite{Rihes} and \cite{Rinv}. Melnikov \cite{Me} described the congruence kernel for $SL_2(\Z)$ proving that $C$ is a free profinite group of countable rank. This description was extended in \cite{Z-05} to all arithmetic subgroups of $SL_2(\R)$.  
In \cite{MPSZ} the description of the congruence kernel $C$  was obtained for arithmetic lattices of an algebraic group ${\bf G}(k_v)$, where $k_v$ is the completion of a global field $k$ with respect to a non-archimeadian valuation $v$. Lubotzky \cite{L} proved that if $O$ is the ring of integers of an imaginary quadratic field then the conguence kernel of $SL_2(O)$  contains a free profinite group of countable rank. In \cite{Lu} Lubotzky also proved that $C$ is finitely generated as a normal subgroup of $\widehat\Gamma$ and Prasad-Rapinchuk \cite{PrRa} showed that it is virtually 1-generated as a normal subgroup of the completion of ${\bf G}(k)$. 

\bigskip
We use Corollary \ref{pro-p of arithmetic}  to give a  full description of finitely generated pro-$p$ subgroups in the congruence kernel of any arithmetic lattice of $SL_2(\C)$ and standard arithmetic lattice of $SO(n,1)$.  

\begin{theorem}\label{congruence kernel}  Let $\Gamma$ be a standard arithmetic lattice of $SO(n,1)$ or any arithmetic lattice of $SL_2(\C)$. Let $H$ be a finitely generated pro-$p$ subgroup of $\widehat\Gamma$. Then $C\cap H$ is free pro-$p$. In particular, every finitely generated pro-$p$ subgroup of the congruence kernel $C$ of $\Gamma$ is free pro-$p$.

\end{theorem}

\subsection*{Acknowledgments} I  thank Andrei Rapinchuk for many valuable discussions on arithmetic groups.

\section{Preliminaries}

 In this section, we recall the necessary elements of  the profinite version of  Bass-Serre's theory. We refer to \cite{R 2017} for further details and to \cite{RZ-00} for the pro-$p$ version of it.

A \emph{graph} $\Gamma$ is a disjoint  union $E(\Gamma) \cup V(\Gamma)$
with two maps $d_0, d_1 : \Gamma \to V(\Gamma)$ that are the
identity on the set of vertices $V(\Gamma)$.  For an element $e$ of
the set of edges  $E(\Gamma)$, $d_0(e) $ is called the \emph{initial} and
$d_1(e) $ the \emph{terminal} vertex of $e$.

\begin{definition}
A \emph{profinite graph} $\Gamma$ is a graph such that:
\begin{enumerate}
\item $\Gamma$ is a profinite space (i.e.\ an inverse limit of finite
discrete spaces);
\item $V(\Gamma)$ is closed; and
\item the maps $d_0$ and $d_1$
are continuous.
\end{enumerate}
Note that $E(\Gamma)$ is not necessary closed.
\end{definition}

A \emph{morphism} $\alpha:\Gamma\longrightarrow \Delta$ of profinite graphs is a continuous map with $\alpha d_i=d_i\alpha$ for $i=0,1$. A profinite graph is called connected if every finite quotient graph of it is connected (as usual graph).

By \cite[Prop.~1.7]{ZM-88} every profinite
graph $\Gamma$ is an inverse limit of finite quotient graphs of
$\Gamma$.

For a profinite space $X$  that is the inverse limit of finite
discrete spaces $X_j$, $[[\widehat{\mathbb{Z}} X]]$ is the inverse
limit  of $ [\widehat{\mathbb{Z}} X_j]$, where
$[\widehat{\mathbb{Z}} X_j]$ is the free
$\widehat{\mathbb{Z}}$-module with basis $X_j$. For a pointed
profinite space $(X, *)$ that is the inverse limit of pointed
finite discrete spaces $(X_j, *)$, $[[\widehat{\mathbb{Z}} (X,
*)]]$ is the inverse limit  of $ [\widehat{\mathbb{Z}} (X_j, *)]$,
where $[\widehat{\mathbb{Z}} (X_j, *)]$ is the free
$\widehat{\mathbb{Z}}$-module with basis $X_j \setminus \{ *
\}$ \cite[Chapter~5.2]{RZ-10}.

For a profinite graph $\Gamma$, define the pointed space
$(E^*(\Gamma), *)$ as  $\Gamma / V(\Gamma)$, with the image of
$V(\Gamma)$ as a distinguished point $*$, and denote the image of $e\in E(\Gamma)$ by $\bar{e}$.  By definition  a profinite
tree  $\Gamma$ is a profinite graph with a short exact sequence
$$
0 \to [[\widehat{\mathbb{Z}}(E^*(\Gamma), *)]]
\stackrel{\delta}{\rightarrow} [[\widehat{\mathbb{Z}} V(\Gamma)]]
\stackrel{\epsilon}{\rightarrow} \widehat{\mathbb{Z}} \to 0
$$
where $\delta(\bar{e}) = d_1(e) - d_0(e)$ for every $e \in E(\Gamma)$ and $\epsilon(v) = 1$ for every $v \in V(\Gamma)$.  If $v$  and $w$ are elements  of a profinite tree   $T$, we denote by $[v,w]$ the smallest profinite subtree of $T$ containing $v$ and $w$ and call it a \emph{geodesic}.

A profinite graph is called a \emph{pro-$p$ tree} if one has the following exact sequence:

$$
0 \to [[\mathbb{F}_p(E^*(\Gamma), *)]]
\stackrel{\delta}{\rightarrow} [[\mathbb{F}_p V(\Gamma)]]
\stackrel{\epsilon}{\rightarrow} \mathbb{F}_p \to 0
$$
where $\delta(\bar{e}) = d_1(e) - d_0(e)$ for every $e \in E(\Gamma)$ and $\epsilon(v) = 1$ for every $v \in V(\Gamma)$. Note that any profinite tree is a pro-$p$ tree, but the converse is not true.  In fact, a connected profinite graph is a profinite tree if and only if it is a pro-$p$ tree for every prime $p$.

By definition, a profinite group $G$ \emph{acts} on a profinite graph
$\Gamma$ if  we have a continuous action of $G$ on the profinite
space $\Gamma$ that commutes with the maps $d_0$ and $d_1$.

 When we say that ${\cal G}$ is a \emph{finite graph of profinite groups}, we mean that it contains the data of the
underlying finite graph, the edge profinite groups, the vertex profinite groups and the attaching continuous maps. More precisely,
let $\Delta$ be a connected finite graph.  The data of a graph of profinite groups $({\cal G},\Delta)$ over
$\Delta$ consists of a profinite group ${\cal G}(m)$ for each $m\in \Delta$, and continuous monomorphisms
$\partial_i: {\cal G}(e)\longrightarrow {\cal G}(d_i(e))$ for each edge $e\in E(\Delta)$.

The definition of the profinite fundamental  group of a connected
profinite graph of profinite groups is quite involved (see
\cite{ZM-89}). However, the profinite fundamental  group
$\Pi_1(\G,\Gamma)$ of a finite graph of finitely generated
profinite groups $(\G, \Gamma)$ can be defined as the profinite completion
of the abstract (usual) fundamental group $\Pi_1^{abs}(\G,\Gamma)$
(we use here that every subgroup of finite index in a finitely
generated profinite group is open \cite[Theorem 1.1 ]{NS-07}).  The fundamental profinite group
$\Pi_1(\G,\Gamma)$
has the following presentation:
\begin{eqnarray}\label{presentation}
\Pi_1(\mathcal{G}, \Gamma)&=&\langle \G(v), t_e\mid rel(\G(v)),
\partial_1(g)=\partial_0(g)^{t_e}, g\in \G(e),\nonumber\\
 & &t_e=1 \  {\rm for}\ e\in
T\rangle;
\end{eqnarray}
here $T$ is a maximal subtree of $\Gamma$ and
$\partial_0:\G(e)\longrightarrow
\G(d_0(e)),\partial_1:\G(e)\longrightarrow \G(d_1(e))$ are
monomorphisms.

In contrast to the abstract case, the vertex groups of $(\G,
\Gamma)$ do not always embed in $\Pi_1(\G,\Gamma)$, i.e.,
$\Pi_1(\G,\Gamma)$ is not always proper. However, the edge and
vertex groups can be replaced by their images in
$\Pi_1(\G,\Gamma)$ and after this replacement $\Pi_1(\G,\Gamma)$
becomes proper. Thus from now on we shall assume that
$\Pi_1(\G,\Gamma)$ is always proper, unless otherwise stated.

To obtain the definition of the fundamental pro-$p$ group of a finite graph of finitely generated pro-$p$ groups one simply replaces `profinite' by `pro-$p$' in the above definition.

The profinite (resp.\ pro-$p$) fundamental  group $\Pi_1(\G,\Gamma)$ acts on the standard profinite (resp.\
pro-$p$) tree $S$ (defined analogously to the abstract one)
associated to it, with vertex and edge stabilizers being conjugates
of vertex and edge groups, and such that
$S/\Pi_1(\G,\Gamma)=\Gamma$  \cite[Proposition 3.8]{ZM-88}.

\begin{example}\label{graph group completion} If $G=\pi_1(\G,\Gamma)$ is the fundamental group of a finite graph of (abstract) groups then one has the induced graph of profinite completions of edge  and vertex groups $(\widehat\G,\Gamma)$ (not necessarily proper) and  a natural homomorphism $G=\pi_1(\G,\Gamma)\longrightarrow \Pi_1(\widehat\G,\Gamma)$. It is an embedding if $\pi_1(\G,\Gamma)$ is residually finite. In this case $\Pi_1(\widehat\G,\Gamma)=\widehat{G}$ is simply the profinite completion.
Moreover,
 the standard tree $S(G)$
  naturally embeds in $S(\widehat G)$ if and only if  the edge and vertex groups $\G(e)$, $\G(v)$     are separable in  $\pi_1(\G,\Gamma)$, or equivalently $\G(e)$ are closed in $\G(d_0(e))$,   $\G(d_1(v))$    with
respect to the topology induced by the profinite topology on $G$
\cite[Proposition 2.5]{CB-13}.  In particular, this is the case  if vertex and edge groups are finitely generated and $G$ is subgroup separable.\end{example}

\begin{definition}\label{def: Profinite acylindrical}
The action of a profinite group $\widehat{\Gamma}$ on a profinite (or pro-$p$) tree $T$ is said to be \emph{$k$-acylindrical}, for $k$ a constant, if the set of fixed points of $\gamma\in\widehat\Gamma$ has diameter at most $k$ whenever $\gamma\neq 1$.
\end{definition}

The next theorem is the main ingredient for proving (ii) of Theorems \ref{main} and \ref{thm: 3M Tits alternative}.

\begin{theorem}\cite[Theorem 3.8]{henry}\label{general acylindrical} Let $G$ be a finitely generated pro-$p$
 group acting $k$-acylindrically on a profinite tree $T$. Then $G$ is a free pro-$p$ product of vertex stabilizers and a free pro-$p$ group.\end{theorem}

\section{Pro-$p$ subgroups of profinite completions of relatively hyperbolic virtually  special groups}

We begin the section with stating necessary definitions and results on relatively hyperbolic and virtually compact special groups.

\begin{definition} A hierarchy of groups of length 0 is a single vertex labeled by a group.
A hierarchy of groups of length $n$ is a graph of groups 
$(\G_n,\Gamma_n)$ together with hierarchies of length $n-1$ on 
each vertex of $\Gamma_n$. If $\curlyH$ is a length $n$ hierarchy of 
groups, the $n$-th level  of $\curlyH$ is the graph of groups 
$(\G_n,\Gamma_n)$.  For $1\leq k\leq n$,  the $(n-k)$-th  level  of $\curlyH$ is  the  disjoint  union  of  the $(n-k)$-th levels of the hierarchies on the vertices of $\Gamma_n$. The terminal groups are the groups labeling the vertices at level 0.
\end{definition}

\begin{theorem}\cite[Theorem 1]{E} \label{einstein} Let $(G,\curlyP)$ be  a   relatively hyperbolic  pair  and  let $G$ be  a  virtually compact  special  group.   Then  there  exists  a  finite  index  subgroup $G_0\leq G$ and  an induced relatively hyperbolic pair $(G_0,\curlyP_0)$ so that $G_0$ has a quasi-convex, malnormal and fully $\curlyP_0$-elliptic hierarchy terminating in groups isomorphic to elements of $\curlyP_0$.
\end{theorem}
\medskip

\begin{remark}\label{separable hierarchy} The Einstein  construction of the hierarchy ensures that the edge groups (the group that you are amalgamating over or taking the HNN extension over) are "full relatively quasi-convex" and hence are separable by \cite[Proposition 7.21]{E}.
\end{remark}

\begin{definition}
Suppose that a group $G$ is hyperbolic relative to a collection of parabolic subgroups $\{P_1,\ldots,P_n\}$.  A subgroup $H$ of $G$ is called \emph{relatively malnormal} if, whenever an intersection of conjugates $H^\gamma\cap H$ is not conjugate into some $P_i$, we have $\gamma\in H$.
\end{definition}

\begin{theorem}(\cite[Theorem 4.2]{WZ-16})\label{thm: Profinite relative malnormality}
Suppose that $G$ is a virtually compact special group, which is also  relatively hyperbolic with parabolic subgroups $P_1,\ldots,P_n$.  Let $H$ be a subgroup which is relatively malnormal and relatively quasi-convex.  Then $\wh{H}$ is also a relatively malnormal subgroup of $\wh{G}$, in the sense that
\[
\wh{H}\cap\wh{H}^{\hat{\gamma}}\leq \wh{P}_i^{\hat g}
\]
(for some $i$ and $\hat g\in \widehat G$) whenever $\wh{\gamma}\notin\wh{H}$.
\end{theorem}

We are now ready to prove Theorem \ref{main}.

\begin{proof}[Proof of Theorem \ref{main}] Let $G$ be a relatively hyperbolic virtually compact special   group. By  Theorem  \ref{einstein} there exists a finite index subgroup $G_0$ of $G$ that admits a quasi-convex malnormal hierarchy terminating in parabolic subgroups and by Remark \ref{separable hierarchy} the edge groups (the group that you are amalgamating over or taking the HNN extension over)  are separable.  Therefore, by Example \ref{graph group completion} $\widehat G_0$ acts on a profinite tree and its edge stabilizers are malnormal by Theorem \ref{thm: Profinite relative malnormality}, i.e. the action is 1-acylidrical. Put $H_0=\widehat G_0\cap H$. 

(i)
 We use  induction on the length of the malnormal hierarchy.   By the inductive hypothesis, the vertex stabilizers in $H_0$ have the claimed structure. Therefore, if $H_0$ stabilizes a vertex by the induction hypothesis we are done. Otherwise, by \cite[Remark 5.3]{WZ-16}, $H_0$ has the claimed structure.

\medskip
(ii) 
 By Theorem \ref{general acylindrical} $H_0$ is a free pro-$p$ product of stabilizers of vertices and a free pro-$p$ group. Applying
Theorem \ref{general acylindrical} inductively to free factors which are vertex stabilizers we end up with the free pro-$p$ product of parabolic pro-$p$  subgroups and a free pro-$p$ group. Then by the main result of \cite{WEZ-16} $H$ is the fundamental group of a finite graph of pro-$p$ groups with finite edge groups  and virtually parabolic vertex groups (the latter follows from its restatement   \cite[Theorem 4.11]{shusterman}).

\end{proof}

\bigskip

Now we turn to the proof of Theorem \ref{relative hypebolization}. Note  that a group is virtually special if and only if it  virtually embeds in Right angled Artin group  (RAAG) and is virtually compact special if and only if it is a virtually virtual retract  of RAAG (see \cite{HW}). We recall a definition of RAAG. Let $\Gamma$ be a finite simplicial graph, and let $V(\Gamma)$ and
$E(\Gamma)$ be the sets of vertices and edges of $\Gamma$
respectively. The right  angled Artin  group $R$, associated to
$\Gamma$, is given by the presentation $ R :=\{V\mid uv = vu\}$,
whenever $u, v \in V$ are ends of an edge $e\in E(\Gamma)$. In the
literature, right  angled Artin groups are also called graph
groups or partially commutative groups.

We shall use it to prove the following lemma.

\begin{lemma}\label{free} Let $G$ be a right angled Artin group  and $F$ a free non-abelian subgroup of $G$. Then the closure $\overline F$ of $F$ in $\widehat G$ contains a free non-abelian pro-$p$ subgroup for some prime $p$. In particular the statement is true for virtually  special groups.

\end{lemma}

\begin{proof} We shall use  a construction of $R=R_n$ as a
sequence of special HNN-extensions $\{1\} = R_0, R_1, . . . , R_n
= R$, i.e. $$R_{i+1}=HNN(R_i, H_i, t)=\langle R_i,t\mid h^t=h, h\in
H_i\rangle,i=0,\ldots, n-1.$$ Note that $R$ is torsion free and
$R_{n-1}$ is a retract (i.e. a semidirect factor) of $R_n$. Hence $R_{n-1}$ is closed in the profinite topology of $R$
and $H_{n-1}$ is also closed by Corollary 9.1 in \cite{Ashot}.

It follows that $\widehat R=HNN(\widehat R_{n-1}, \overline
H_{n-1}, t)$ and a standard tree $S(R)$ on which $R$ acts embeds
in the standard profinite tree $S(\widehat R)$ on which $\widehat
R$ acts (see Example \ref{graph group completion}). In this case
 the vertex set is $V(S(\widehat R))= \displaystyle \widehat R/\widehat R_{n-1}$,
  the edge set is $E(S(\widehat R))= \widehat R/\widehat H_{n-1}$, and
  the initial and terminal vertices of an edge $g\widehat H_{n-1}$ are
  respectively  $g\widehat R_{n-1}$ and $gt\widehat R_{n-1}$. The tree
  $S(R)$ is defined similarly.  We use induction on $n$. For $n=1$ the result is clear. Suppose the result holds for $n-1$.  Consider the action of $G$ on $S(R)$.  If $\overline F$ does not contain a free non-abelian pro-$p$ subgroup then by the main result of \cite{Zal90} or \cite[Theorem 4.2.11]{R 2017} either $\overline F$ fixes a vertex $v$ or there exists an edge $e$ such that $\overline F_e$ is normal and  $\overline F/\overline F_e$ is soluble. So either $F_v$ or $F_e$ contains a finitely generated free subgroup. Hence   the result follows from the induction hypothesis. 
  
  To prove the last statement of the theorem we may assume w.l.o.g. that $G$ is a subgroup of $R$ and apply the first statement taking into account that if the closure of $F$ in $\overline G$ contains a non-abelian free pro-$p$ subgroup then so does the closure of $F$ in $\widehat G$.

\end{proof} 

\begin{theorem} A virtually compact special group $G$ is  hyperbolic relative to a finite collection of virtually abelian subgroups if and only if $\widehat G$ does not contain $F_2\times \Z_p$, where $F_2$ is a free pro-$p$ group of rank 2.

\end{theorem}

\begin{proof} Let $G$ be a toral hyperbolic virtually compact special
group. Then by Theorem \ref{main} any subgroup of $\widehat G$ isomorphic to $\Z_p\times \Z_p$  is conjugate into the profinite completion of a maximal parabolic subgroup. Since the profinite completion of a maximal parabolic subgroup is malnormal by \cite[Theorem 3.3]{WZ-16}  one deduces that $F_2\times \Z_p$ must be a subgroup of the profinite completion of a parabolic subgroup, contradicting the hypothesis.

Conversely, suppose $G$ is virtually compact special not toral hyperbolic. Then by \cite[Theorem 1.4]{G} $G$ contains $F\times \Z$, where $F$ is free non-abelian. Let $H$ be the closure of $F\times \Z$ in $\widehat G$. Since the center of $F\times \Z$ is cyclic and cyclic groups of virtually compact special groups are virtual retracts (see \cite[Corollary 1.6]{M}) and so are separable,  the center $Z(H)$ of $H$ contains $\widehat \Z$. By Lemma \ref{free} $H$ contains a free non-abelian pro-$p$ subgroup. Hence $H$ contains  $F_2\times \Z_p$. 
\end{proof}

\section{Arithmetic hyperbolic manifolds} 

The closed hyperbolic case is the subject of \cite[Theorem F]{WZ-16};  in this
case our subgroup is free pro-$p$.
Thus we consider cusped hyperbolic standard arithmetic manifolds in this subsection.

Recall that a subgroup of $\pi_1M$  is called \emph{peripheral}  if it is conjugate to the fundamental group of a cusp and so is virtually isomorphic to $\Z^{n-1}$. A conjugate of the closure of a peripheral group in $\widehat{\pi_1M}$ will also be called peripheral; it is virtually  isomorphic to $\widehat \Z^{n-1}$.
It is well known that $\pi_1M$ is hyperbolic relative to its peripheral subgroups.

\bigskip
{\it Proof of Theorem \ref{thm: 3M Tits alternative}}. By \cite[Theorem 1.10]{bergeron} in cocompact case (resp.  by \cite[Lemma 6.3 and Theorem 7.4]{PS} in not cocompact case) $\pi_1M$ is virtually compact special (resp. a virtually  virutal retract of a virtually compact special group) and as mentioned above  hyperbolic relative to the peripheral subgroups which are virtually abelian. Therefore the result is a consequence of Theorem \ref{main}.

\begin{remark} Every virtually abelian free factor of Theorem \ref{thm: 3M Tits alternative} (ii) is a pro-$p$ Bieberbach group, i.e. the pro-$p$ completion of a torsion free extension of a free abelian group by finite $p$-group (see the last paagraph of \cite{camina}).  

\end{remark}

\bigskip
{\it Proof of Theorem \ref{congruence kernel}}.  Replacing $\Gamma$ by a  congruence subgroup if necessary,  we may assume w.l.o.g. that $\Gamma$ is torsion free and all peripheral subgroups are abelian. 

By Theorem \ref{thm: 3M Tits alternative} in the case of a standard arithmetic lattice of $SO(n,1)$ and by \cite[Theorem 4.6]{henry} (and its proof) in the case of any lattice in $SL_2(\C)$,  a finitely generated pro-$p$ subgroup $H$ of $\widehat \Gamma$ is a free product of abelian pro-$p$ groups and all free abelian factors of rank $>1$ are conjugate to cusp subgroups of $\widehat G$. But the congruence topology induces the (full) profinite topology on abelian subgroup (see \cite[Theorem 5 on page 61]{segal} combined with the fact that the congruence topology does not depend on a representation \cite[Lemma on page 301]{raghunatan}). Thus the congruence kernel does not intersect the cusp subgroups, so by the pro-$p$ version of the Kurosh subgroup theorem (see \cite[Theorem 9.6.2]{R 2017}) $C\cap H$ is a free pro-$p$ product of cyclic pro-$p$ groups and therefore is free pro-$p$.

\bibliographystyle{alpha}
%\bibliography{/Users/henrywilton/Dropbox/library}

%\iffalse

\end{document}